\documentclass[11pt,a4paper]{amsart}
\usepackage{amsfonts,amsmath,amssymb,amsthm}
\usepackage{graphics,graphicx}
\usepackage{enumerate}
\usepackage[utf8]{inputenc}
\usepackage{bbm, dsfont}

\usepackage{mathtools}
\mathtoolsset{showonlyrefs}

\numberwithin{equation}{section}

\DeclareSymbolFont{SY}{U}{psy}{m}{n}
\DeclareMathSymbol{\emptyset}{\mathord}{SY}{'306}

\newtheorem{thm}{Theorem}{\bf}{\it}
\newtheorem{lem}[thm]{Lemma}{\bf}{\it}
\newtheorem{remark}[thm]{Remark}{\bf}{\it}

\theoremstyle{definition}

\title[Another proof of Moon's theorem on generalised tournament score sequences]{Another proof of Moon's theorem on generalised tournament score sequences}
\author[Erik ~Th\"ornblad]{Erik Th\"ornblad}
 \address{Department of Mathematics, Uppsala University, Box 480, S-75106 Uppsala, Sweden.}
 \email{erik.thornblad@math.uu.se}
 \date{\today}

\begin{document}
\begin{abstract}
Landau \cite{Landau1953} showed that a sequence $(d_i)_{i=1}^n$ of integers is the score sequence of some tournament if and only if $\sum_{i\in J}d_i \geq \binom{|J|}{2}$ for all $J\subseteq \{1,2,\dots, n\}$, with equality if $|J|=n$. Moon \cite{Moon63} extended this result to generalised tournaments. We show how Moon's result can be derived from Landau's result.
\end{abstract}
\maketitle

\section{Introduction}
A \emph{tournament} is a directed complete graph. Formally, it is a graph $G=(V(G),E(G))$ such that $(i,j)\in E(G)$ if and only if $(j,i)\notin E(G)$ for all distinct $i,j\in V(G)$, and moreover $(i,i)\notin E(G)$ for all $i\in V(G)$. Given a vertex $u\in V(G)$, its outdegree is defined as the sum $\sum_{v\in V(G)} \mathbbm{1}((u,v)\in E(G))$, i.e. the number of outgoing edges from $u$. These numbers can be computed for all vertices of $G$, resulting in the sequence $(d_i)_{i=1}^n$ of outdegrees of $G$, known as the \emph{score sequence} of $G$. A classical result is the following characterisation of permissible score sequences, originally due to Landau \cite{Landau1953}.

\begin{thm}[\cite{Landau1953}]\label{thm:landau}
A sequence $(d_i)_{i=1}^n$ of non--negative integers is the score sequence of some tournament on $n$ vertices if and only if $\sum_{i\in J} d_i\geq {|J| \choose 2}$ for all subsets $J\subseteq \{1,2,\dots , n\}$, with equality for $J=\{1,2,\dots, n\}$.
\end{thm}

It is easy to see necessity in the above theorem. Namely, given any subset of size $k$ of the vertices, look at the induced subtournament. This must have $\binom{k}{2}$ internal edges, contributing this much to the sum of the scores of the vertices in this subtournament, and possibly some edges leaving from this subtournament to the rest of the tournament. For sufficiency there are numerous proofs, see e.g. \cite{GriggsReid1999, Landau1953, Thomassen1981}.

A \emph{generalised tournament} $G=(V(G),\alpha)$ is a set $V(G)=\{1,\dots, n\}$ of vertices along with a function $\alpha:V(G)\times V(G) \to [0,1]$, such that $\alpha(i,j)+\alpha(j,i)=1$ for all $(i,j)\in V(G)\times V(G)$, $i\neq j$, and $\alpha(i,i)=0$ for all $i\in V(G)$. We will sometimes call the function $\alpha$ the edge--weight function, etc. Note that a tournament is a generalised tournament for which $\alpha$ only takes values in $\{0,1\}$. Given a vertex $u\in V$, the \emph{outdegree} of $u$ is defined as the sum $\sum_{v\in V(G)}\alpha(u,v)$, extending the definition for tournaments to generalised tournaments. Score sequences are defined analogously. A natural question is whether Theorem \ref{thm:landau} extends to generalised tournaments. This was answered in the affirmative by Moon \cite{Moon63}.

\begin{thm}[\cite{Moon63}] \label{thm:moon}
 A sequence $(d_i)_{i=1}^n$ of non--negative real numbers is the score sequence of some generalised tournament if and only if $\sum_{i\in J} d_i\geq {|J| \choose 2}$ for all subsets $J\subseteq \{1,2,\dots , n\}$, with equality for $J=\{1,2,\dots, n\}$. 
\end{thm}

The proof given by Moon \cite{Moon63} uses a network flow approach by showing that the existence of a generalised tournament with the prescribed score sequence follows from the existence of a flow satisfying certain conditions. Another proof of Theorem \ref{thm:moon} is given by Bang and Sharp \cite{BangSharp1977}, in which the authors use a version of Hall's theorem along with a rational approximation and compactness argument. The same kind of argument lies at the heart of our proof as well; however, rather than using Hall's theorem, we will assume only the validity of Theorem \ref{thm:landau}. Since Theorem \ref{thm:landau} is precisely the tournament--version of Theorem \ref{thm:moon}, this is perhaps the most ``natural'' starting point. Thus, in this note, we present a new proof of Theorem \ref{thm:moon}, assuming only the validity of \ref{thm:landau}.

\subsection{Outline}
In this section we outline how Theorem \ref{thm:moon} follows from Theorem \ref{thm:landau}. The proofs appear in Section \ref{sec:proofs}. 

Our approach will be to argue via rational approximations. To this end, we begin by showing that Theorem \ref{thm:landau} extends to generalised tournaments with rational score sequences.

\begin{lem}\label{lem:rational}
A sequence $(d_i)_{i=1}^n$ of non--negative rationals is the score sequence of some generalised tournament on $n$ vertices if and only if $\sum_{i\in J} d_i\geq {|J| \choose 2}$ for all subsets $J\subseteq \{1,2,\dots , n\}$, with equality for $J=\{1,2,\dots, n\}$.
\end{lem}

The forward direction is straightforward, so we prove only the backward direction. The idea of the proof is the following. Start with a rational sequence $(d_i)_{i=1}^n$ as above. We want to prove that there is a generalised tournament having this sequence as its score sequence. Instead we will consider a related sequence on $mn$ elements, all integers, where $m$ is chosen large enough that all $md_i$ are integers. We will show that this sequence satisfies the condition in Theorem \ref{thm:landau}, so that it is the score sequence of some non--generalised tournament. Having this tournament, we will group its $mn$ vertices into $n$ clusters. Then we define a generalised tournament on $n$ vertices. Each vertex will correspond to one of the clusters, and the edge weights between vertices will be the average edge weight between the corresponding clusters. It suffices to show that the object defined in this manner is a well--defined generalised tournament with score sequence $(d_i)_{i=1}^n$.

\begin{remark}Given a score sequence, it is natural to ask for an algorithm that outputs a tournament with the given score sequence, see e.g. \cite{Gervacio1988}. Any such algorithm, coupled with the procedure in our proof, can also be used to construct generalised tournament with given rational score sequence. Moreover, the edge weights in any such construction will be rational.
\end{remark}

To prove Theorem \ref{thm:moon}, we need to approximate the real score sequence by rational score sequences. The following lemma states that this can be done in the desired way.
\begin{lem}\label{lem:tech}
 Let $(d_i)_{i=1}^n$ be a sequence of non--negative reals such that 
\begin{align}
 \sum_{i\in J} d_i \geq \binom{|J|}{2} 
\end{align}
for any $J\subseteq \{1,2,\dots, n\}$ with equality for $J=\{1,2,\dots, n\}$. Then there exist sequences $(d_i^{(m)})_{i=1}^n$ of non--negative rationals such that $d_i^{(m)} \to d_i$ as $m\to \infty$, for each $i=1,2, \dots, n$, and moreover, for all $m\geq 1$,
\begin{align}
 \sum_{i\in J} d_i^{(m)} \geq \binom{|J|}{2} 
\end{align}
for any $J\subseteq \{1,2,\dots, n\}$ with equality for $J=\{1,2,\dots, n\}$. \end{lem}

The idea behind the proof is to perturb the sequence $(d_i)_{i=1}^n$ to a rational sequence without disturbing the validity of the above condition.

Given Lemma \ref{lem:rational} and Lemma \ref{lem:tech}, we then show that Theorem \ref{thm:moon} holds. The proof of this is not difficult, and we only prove the more difficult direction. Given a real sequence $(d_i)_{i=1}^n$ satisfying the condition in Theorem \ref{thm:moon}, we will approximate it by rational sequences $(d_i^{(m)})_{i=1}^n$ as in Lemma \ref{lem:tech}. By Lemma \ref{lem:rational} there exists some generalised tournament $G_m$ with score sequence $(d_i^{(m)})_{i=1}^n$; we may furthermore assume that these are defined on the same vertex set. Finally we note that the edge weights of the $G_m$ form a sequence in the compact set $[0,1]^{\binom{n}{2}}$, so we may select a subsequence of $(G_m)_{m=1}^{\infty}$ such that all edge weights converge. The limiting object will turn out to be a generalised tournament with score sequence $(d_i)_{i=1}^n$.

The proof method outlined above should allow for similar extensions from (some) ``non--generalised'' results to their ``generalised'' counterparts. Such an example is provided in \cite{Thornblad2016e}, which uses the same proof method to show that Eplett's characterisation of possible score sequences of self--converse tournaments also is valid for self--converse generalised tournaments.

\section{Proofs}\label{sec:proofs}

\begin{proof}[Proof of Lemma \ref{lem:rational}]
  Let $(d_i)_{k=1}^n$ be a sequence of non--negative rational numbers as in the statement of the lemma. Since the $d_i$ are rational, there exist $k_i,m_i\in \mathbb{N}$ (with no common factors) such that $d_i=k_i/m_i$. Denote by $m$ the lowest common multiple of $m_i$. (If some $k_i=0$, we may take $m_i=1$; this may happen for at most one $i$.)
 
We first construct an $n\times m$--array which will contain the outdegrees of a (non--generalised) tournament. For $i=1,\dots, n$ and $\ell=1,\dots, m$, let
\begin{align}
c_{i,\ell} = md_i+b_{\ell}
\end{align}
where $(b_{\ell})_{\ell=1}^m$ is some arbitrary score sequence of a tournament on $m$ vertices. Note that $c_{i,\ell}$ is an integer for all $i=1,\dots, n$ and $\ell=1,\dots, m$. We claim that the $c_{i,\ell}$ constitutes a valid tournament score sequence, i.e. that the condition in Theorem \ref{thm:landau} holds. To see this, take some $J\subseteq \{1,\dots, n\}\times \{1,\dots, m\}$ and consider the partition $J=J_1\cup J_2\cup \dots \cup J_n$, where $J_i=\{(i,k)\in J \ : \ k\in\{1,2,\dots, m\}\}$. Let $j_i=|J_i|$ and assume that the $J_i$ have been ordered so that $0\leq j_1\leq j_2 \dots \leq j_n \leq m$. It suffices to check the three inequalities
\begin{align}
 \sum_{i=1}^n j_id_i & \geq \sum_{i=1}^n j_i(n-i), \\
 \sum_{(i,\ell)\in J} b_{i,\ell} & \geq \frac{1}{2}\left(\sum_{i=1}^n j_i^2 -\sum_{i=1}^n j_i \right), \\
m\sum_{i=1}^n j_i(n-i) & \geq \frac{1}{2} \left(\left(\sum_{i=1}^n j_i\right)^2 -\sum_{i=1}^n j_i^2 \right).
\end{align}
To see that this is enough, note that
\begin{align}
 \sum_{(i,\ell)\in J}c_{i,\ell}=m\sum_{i=1}^n \sum_{(i,\ell)\in J_i}d_i + \sum_{(i,\ell)\in J}b_{i,\ell} 
&= m\sum_{i=1}^n j_id_i + \sum_{(i,\ell)\in J}b_{i,\ell} \\ 
&\geq \frac{1}{2}\left(\left(\sum_{i=1}^n j_i\right)^2-\sum_{i=1}^n j_i \right) \\
& = \binom{|J|}{2},
\end{align}
where the inequality follows by the three inequalities above, and the final equality follows by $|J|=\sum_{i=1}^n j_i$. 

Let us verify that the three inequalities hold. Define for notational convention $j_0=0$. The first inequality follows by noting that we can rewrite the sum (twice) as a telescoping sum and using the fact that the sequence $(d_i)_{i=1}^n$ satisfies our condition, i.e.
\begin{align}
 \sum_{i=1}^n j_id_i = \sum_{i=1}^n (j_i-j_{i-1})\sum_{k=i}^n d_k
&\geq \sum_{i=1}^n (j_i-j_{i-1})\binom{n-i+1}{2} \\
&=\sum_{i=1}^n j_i \left(\binom{n-i+1}{2}-\binom{n-i}{2} \right) \\
&=\sum_{i=1}^n j_i(n-i)
\end{align}
The second inequality follows since $(b_{\ell})_{\ell=1}^m$ forms a valid score sequence, so
\begin{align}
 \sum_{(i,\ell)\in J}b_{i,\ell} = \sum_{i=1}^n \sum_{(i,\ell)\in J_i}b_{i,\ell} \geq \sum_{i=1}^n \binom{j_i}{2} =  \frac{1}{2}\left(\sum_{i=1}^n j_i^2 -\sum_{i=1}^n j_i \right)
\end{align}
For the third inequality,
\begin{align}
 m\sum_{i=1}^{n} j_i(n-i) = m\sum_{i=1}^{n-1} \sum_{r=i+1}^n j_i \geq \sum_{i=1}^{n-1} \sum_{r=i+1}^n j_i j_r = \frac{1}{2} \left(\left(\sum_{i=1}^n j_i\right)^2 -\sum_{i=1}^n j_i^2 \right).
\end{align}


This shows that the $c_{i,\ell}$ form a valid score sequence. By Theorem \ref{thm:landau}, there exists a (non--generalised) tournament $H$ on $mn$ vertices with outdegrees $c_{i,\ell}$. Denote by $v_{i,\ell}$ the vertex of $H$ with outdegree $c_{i,\ell}$, so that 
\begin{align}
 V(H)=\{v_{i,\ell} \ : \ i=1,\dots, n \text{ and } \ell=1,\dots, m\}.
\end{align}

Define the function $\alpha : \{1,2,\dots, n\}^2 \to [0,1]$ by 
\begin{align}
 \alpha(i,j) = \frac{1}{m^2}\sum_{\ell=1}^m \sum_{k=1}^m \mathbbm{1}((v_{i,\ell},v_{j,k})\in E(H))
\end{align}
for any $i,j=1,2,\dots, n$ with $i\neq j$, and $\alpha(i,i)=0$ for all $i=1,2,\dots, n$. We claim that $G=(V(G),\alpha)$, where $V(G)=\{1,2,\dots, n\}$, is a well--defined generalised tournament with score sequence $(d_i)_{i=1}^n$. It is well--defined since
\begin{align}
 \alpha(i,j)+\alpha(j,i) = \frac{1}{m^2} \sum_{\ell=1}^m \sum_{k=1}^m (\mathbbm{1}((v_{i,\ell},v_{j,k})\in E(H)) + \mathbbm{1}(v_{j,k},v_{i,\ell})\in E(H))) = 1,
\end{align}
for all $i,j=1,2,\dots, n$ with $i\neq j$. Trivially also $\alpha(i,j)\in [0,1]$. To see that $G$ has score sequence $(d_i)_{i=1}^n$, note that
\begin{align}
 \sum_{\substack{j=1 \\ j\neq i}}^n \alpha(w_i,w_j) 
& = \frac{1}{m^2} \sum_{\ell=1}^m \sum_{\substack{j=1 \\ j\neq i}}^n \sum_{k=1}^m \mathbbm{1}((v_{i,\ell},v_{j,k})\in E(H)) \\
& = \frac{1}{m^2} \sum_{\ell=1}^m (md_i+b_{\ell}) - \frac{1}{m^2} \binom{m}{2} \\
&= d_i,
\end{align}
where we have cancellation because $b_{\ell}$ was chosen to be a valid score sequence, implying that $\sum_{\ell=1}^m b_{\ell}=\binom{m}{2}$. This proves that $G$ is a generalised tournament (with rational edge weights) with score sequence $(d_i)_{i=1}^n$, completing the proof.
\end{proof}

\begin{proof}[Proof of Lemma \ref{lem:tech}]
Fix $m\geq 1$. Suppose the numbers $(d_i)_{i=1}^n$ have been ordered so that $d_1\geq \max\{d_2,\dots, d_n\}$. We claim first that this implies
\begin{align}
 d_1+\sum_{i\in J}d_i > \binom{|J|+1}{2}
\end{align}
for all $|J|\subsetneqq \{2,\dots, n\}$. Suppose not and let the set $J\subsetneqq \{2,\dots, n\}$ be a counterexample. Take any $1\neq k\notin J$. Then
\begin{align}
 \binom{|J|+1}{2} = d_1 + \sum_{i\in J}d_i \geq d_k + \sum_{\in J}d_i \geq \binom{|J|+1}{2}
\end{align}
which implies that $d_k=d_1$. Then
\begin{align}
\binom{|J|+2}{2} \leq  d_1+d_k+\sum_{i \in J}d_i 
&= \left( d_1+\sum_{i\in J}d_i\right) + \left(d_k + \sum_{i\in J}d_i\right) - \sum_{i\in J}d_i \\
&\leq 2\binom{|J|+1}{2}-\binom{|J|}{2} \\
&= \binom{|J|+2}{2}-1,
\end{align}
a contradiction. 

The above implies that we can choose $d_{1}^{(m)}$ to be some non--negative rational such that $d_{1}-1/m<d_{1}^{(m)}<d_{1}$ and
\begin{align}
 d_1^{(m)}+ \sum_{i\in J} d_i  > \binom{|J|}{2}
\end{align}
for all $J \subsetneqq \{2,\dots, n\}$. This gives us enough room to define $d_i^{(m)}$ for all other $i=2,3,\dots, n$. We choose these as follows. For $i=2,3,\dots, n-1$, let $d_i^{(m)}$ be some rational satisfying
\begin{align}
 d_{i}< d_{i}^{(m)}<d_{i}+\frac{d_1-d_1^{(m)}}{n-1}
\end{align}
Finally let
\begin{align}
 d_n^{(m)} = \binom{n}{2} - \sum_{i=1}^{n-1} d_i^{(m)}.
\end{align}
Note that
\begin{align}
 d_n^{(m)} = \binom{n}{2} - \sum_{i=1}^{n-1} d_i^{(m)} > \binom{n}{2} - d_1^{(m)} - \sum_{i=2}^{n-1} d_i - (d_1-d_1^{(m)}) = \binom{n}{2} -\sum_{i=1}^{n-1} d_i = d_n.
\end{align}

With this choice of $(d_i^{(m)})_{i=1}^n$, we have 
\begin{align}
 \sum_{i\in J}d_i^{(m)} \geq \binom{|J|}{2} 
\end{align}
for all $J\subseteq \{1,2,\dots, n\}$ with equality for $J=\{1,2,\dots, n\}$. To see this, note that $d_i^{(m)}>d_i$ for $i=2,3,\dots, n$, and since the corresponding inequalities hold for $(d_i)_{i=1}^n$, there can only be a problem if $1\in J$. But this was taken care of by our choice of $d_1^{(m)}$. By construction we also have $d_i^{(m)}\to d_i$ as $m\to \infty$, for all $i=1,2, \dots, n$.
\end{proof}

\begin{proof}[Proof of Theorem \ref{thm:moon}]
Let $(d_i)_{i=1}^n$ be a sequence of non--negative reals satisfying 
\begin{align}
 \sum_{i\in J}d_i^{(m)} \geq \binom{|J|}{2} 
\end{align}
for all $J\subseteq \{1,2,\dots, n\}$ with equality for $J=\{1,2,\dots, n\}$. By Lemma \ref{lem:tech}, we can approximate $(d_i)_{i=1}^n$ by rational sequences $(d_i^{(m)})_{i=1}^n$ such that $d_i^{(m)}\to d_i$ as $m\to \infty$, and such that $(d_i^{(m)})_{i=1}^m$ is the score sequence of some generalised tournament (applying Lemma \ref{lem:rational}). Denote by $G_m=(V(G_m),\alpha_m)$ any generalised tournament with score sequence $(d_i^{(m)})_{i=1}^m$. Assume that $V(G_m)=\{1,2,\dots, n\}$ for each $m\geq 1$, so that we may identify the edge sets.

By compactness of $[0,1]^{\binom{n}{2}}$, we may after passing to a subsequence assume that the tuple $(\alpha_m(i,j))_{i,j\in V(G)}$ converges coordinatwise.  Let $\alpha \ : \ \{1,\dots, n\} \to \{1,\dots, n\}$ the function defined by
\begin{align}
 \alpha(i,j) = \lim_{m\to \infty} \alpha_m(i,j)
\end{align}
for any $i,j=1,2,\dots, n$. Note that $\alpha(i,j)+\alpha(j,i)=1$ for $i\neq j$ and $\alpha(i,j)\in [0,1]$, so $\alpha$ defines a weight function. Let $G=(V(G),\alpha)$, where $V(G)=\{1,2,\dots, n\}$ be the generalised tournament with weight function $\alpha$. Since
\begin{align}
 \sum_{\substack{i=1 \\ i\neq j}}^n \alpha(i,j) = \lim_{m\to \infty} \sum_{\substack{i=1 \\ i\neq j}}^n \alpha_m(i,j) = \lim_{m\to \infty}d_i^{(m)}=d_i,
\end{align}
the generalised tournament $G$ has score sequence $(d_i)_{i=1}^n$. This proves Theorem \ref{thm:moon}.
\end{proof}

\section*{Acknowledgements}
The author thanks Katja Gabrysch for figuring out how to verify that the numbers $c_{i,\ell}$ form a valid score sequence, and an anonymous referee for helpful comments.

\bibliographystyle{abbrv}

\end{document}